\documentclass{amsart}
\usepackage{amsthm}
\usepackage{amssymb}
\usepackage{amsmath}
\usepackage{amsmath,amscd}
\usepackage{color}
\usepackage{hyperref}
\usepackage[all,arc]{xy}
\usepackage{pdfsync}
\usepackage[title]{appendix}

\newtheorem{thm}{Theorem}[section]
\newtheorem{cor}[thm]{Corollary}
\newtheorem{prop}[thm]{Proposition}
\newtheorem{lemma}[thm]{Lemma}

\theoremstyle{definition}
\newtheorem{defin}[thm]{Definition}

\theoremstyle{remark}
\newtheorem{rem}[thm]{Remark}

\newcommand{\nc}{\newcommand}
\nc{\rnc}{\renewcommand}
\nc{\bb}[1]{{\mathbb #1}}
\nc{\bbA}{\bb{A}}\nc{\bbB}{\bb{B}}\nc{\bbC}{\bb{C}}\nc{\bbD}{\bb{D}}
\nc{\bbE}{\bb{E}}\nc{\bbF}{\bb{F}}\nc{\bbG}{\bb{G}}\nc{\bbH}{\bb{H}}
\nc{\bbI}{\bb{I}}\nc{\bbJ}{\bb{J}}\nc{\bbK}{\bb{K}}\nc{\bbL}{\bb{L}}
\nc{\bbM}{\bb{M}}\nc{\bbN}{\bb{N}}\nc{\bbO}{\bb{O}}\nc{\bbP}{\bb{P}}
\nc{\bbQ}{\bb{Q}}\nc{\bbR}{\bb{R}}\nc{\bbS}{\bb{S}}\nc{\bbT}{\bb{T}}
\nc{\bbU}{\bb{U}}\nc{\bbV}{\bb{V}}\nc{\bbW}{\bb{W}}\nc{\bbX}{\bb{X}}
\nc{\bbY}{\bb{Y}}\nc{\bbZ}{\bb{Z}}
\nc{\mbf}[1]{{\mathbf #1}}
\nc{\bfA}{\mbf{A}}\nc{\bfB}{\mbf{B}}\nc{\bfC}{\mbf{C}}\nc{\bfD}{\mbf{D}}
\nc{\bfE}{\mbf{E}}\nc{\bfF}{\mbf{F}}\nc{\bfG}{\mbf{G}}\nc{\bfH}{\mbf{H}}
\nc{\bfI}{\mbf{I}}\nc{\bfJ}{\mbf{J}}\nc{\bfK}{\mbf{K}}\nc{\bfL}{\mbf{L}}
\nc{\bfM}{\mbf{M}}\nc{\bfN}{\mbf{N}}\nc{\bfO}{\mbf{O}}\nc{\bfP}{\mbf{P}}
\nc{\bfQ}{\mbf{Q}}\nc{\bfR}{\mbf{R}}\nc{\bfS}{\mbf{S}}\nc{\bfT}{\mbf{T}}
\nc{\bfU}{\mbf{U}}\nc{\bfV}{\mbf{V}}\nc{\bfW}{\mbf{W}}\nc{\bfX}{\mbf{X}}
\nc{\bfY}{\mbf{Y}}\nc{\bfZ}{\mbf{Z}}
\nc{\bfa}{\mbf{a}}\nc{\bfb}{\mbf{b}}\nc{\bfc}{\mbf{c}}\nc{\bfd}{\mbf{d}}
\nc{\bfe}{\mbf{e}}\nc{\bff}{\mbf{f}}\nc{\bfg}{\mbf{g}}\nc{\bfh}{\mbf{h}}
\nc{\bfi}{\mbf{i}}\nc{\bfj}{\mbf{j}}\nc{\bfk}{\mbf{k}}\nc{\bfl}{\mbf{l}}
\nc{\bfm}{\mbf{m}}\nc{\bfn}{\mbf{n}}\nc{\bfo}{\mbf{o}}\nc{\bfp}{\mbf{p}}
\nc{\bfq}{\mbf{q}}\nc{\bfr}{\mbf{r}}\nc{\bfs}{\mbf{s}}\nc{\bft}{\mbf{t}}
\nc{\bfu}{\mbf{u}}\nc{\bfv}{\mbf{v}}\nc{\bfw}{\mbf{w}}\nc{\bfx}{\mbf{x}}
\nc{\bfy}{\mbf{y}}\nc{\bfz}{\mbf{z}}

\newcommand{\h}{\mathrm{ht}}

\newcommand{\C}{{\mathbb C}}

\newcommand{\Z}{{\mathbb Z}}

\newcommand{\cO}{{\mathcal O}}

\nc{\mcal}[1]{{\mathcal #1}}
\nc{\calA}{\mcal{A}}\nc{\calB}{\mcal{B}}\nc{\calC}{\mcal{C}}\nc{\calD}{\mcal{D}}
\nc{\calE}{\mcal{E}} \nc{\calF}{\mcal{F}}\nc{\calG}{\mcal{G}}\nc{\calH}{\mcal{H}}
\nc{\calI}{\mcal{I}}\nc{\calJ}{\mcal{J}}\nc{\calK}{\mcal{K}}\nc{\calL}{\mcal{L}}
\nc{\calM}{\mcal{M}}\nc{\calN}{\mcal{N}}\nc{\calO}{\mcal{O}}\nc{\calP}{\mcal{P}}
\nc{\calQ}{\mcal{Q}}\nc{\calR}{\mcal{R}}\nc{\calS}{\mcal{S}}\nc{\calT}{\mcal{T}}
\nc{\calU}{\mcal{U}}\nc{\calV}{\mcal{V}}\nc{\calW}{\mcal{W}}\nc{\calX}{\mcal{X}}
\nc{\calY}{\mcal{Y}}\nc{\calZ}{\mcal{Z}}
\nc{\fA}{\frak{A}}\nc{\fB}{\frak{B}}\nc{\fC}{\frak{C}} \nc{\fD}{\frak{D}}
\nc{\fE}{\frak{E}}\nc{\fF}{\frak{F}}\nc{\fG}{\frak{G}}\nc{\fH}{\frak{H}}
\nc{\fI}{\frak{I}}\nc{\fJ}{\frak{J}}\nc{\fK}{\frak{K}}\nc{\fL}{\frak{L}}
\nc{\fM}{\frak{M}}\nc{\fN}{\frak{N}}\nc{\fO}{\frak{O}}\nc{\fP}{\frak{P}}
\nc{\fQ}{\frak{Q}}\nc{\fR}{\frak{R}}\nc{\fS}{\frak{S}}\nc{\fT}{\frak{T}}
\nc{\fU}{\frak{U}}\nc{\fV}{\frak{V}}\nc{\fW}{\frak{W}}\nc{\fX}{\frak{X}}
\nc{\fY}{\frak{Y}}\nc{\fZ}{\frak{Z}}
\nc{\fa}{\frak{a}}\nc{\fb}{\frak{b}}\nc{\fc}{\frak{c}} \nc{\fd}{\frak{d}}
\nc{\fe}{\frak{e}}\nc{\fFf}{\frak{f}}\nc{\fg}{\frak{g}}\nc{\fh}{\frak{h}}
\nc{\fri}{\frak{i}}\nc{\fj}{\frak{j}}\nc{\fk}{\frak{k}}\nc{\fl}{\frak{l}}
\nc{\fm}{\frak{m}}\nc{\fn}{\frak{n}}\nc{\fo}{\frak{o}}\nc{\fp}{\frak{p}}
\nc{\fq}{\frak{q}}\nc{\fr}{\frak{r}}\nc{\fs}{\frak{s}}\nc{\ft}{\frak{t}}
\nc{\fu}{\frak{u}}\nc{\fv}{\frak{v}}\nc{\fw}{\frak{w}}\nc{\fx}{\frak{x}}
\nc{\fy}{\frak{y}}\nc{\fz}{\frak{z}}

\DeclareMathOperator{\pt}{pt}

\DeclareMathOperator{\End}{End}

\DeclareMathOperator{\RHom}{RHom}
\DeclareMathOperator{\val}{val}
\DeclareMathOperator{\Ind}{Ind}

\begin{document}

\title{Whittaker functions from motivic Chern classes}

\date{\today}

\author[Leonardo C. Mihalcea]{Leonardo C.~Mihalcea}
\address{
Department of Mathematics, 
Virginia Tech University, 
Blacksburg, VA USA 24061
}
\email{lmihalce@vt.edu}
\thanks{L.C. Mihalcea was supported by a Simons Collaboration Grant}

\author[Changjian Su]{Changjian Su, \\ with an Appendix joint with Dave Anderson}
\address{Department of Mathematics, University of Toronto, Toronto, ON, Canada}
\email{changjiansu@gmail.com}

\address{Department of Mathematics, The Ohio State University, Columbus, OH USA 43210}
\email{anderson.2804@math.osu.edu}

\begin{abstract} We prove a `motivic' analogue of the Weyl character formula, computing the Euler 
characteristic of a line bundle on a generalized flag manifold $G/B$ multiplied either by a motivic 
Chern class of a Schubert cell, or a Segre analogue of it. The result, given in terms of 
Demazure-Lusztig (D-L) operators, {identifies an Euler characteristic above to a formula of Brubaker, Bump and 
Licata for the Iwahori-Whittaker functions of the principal series representation of the $p$-adic Langlands 
dual group. As a corollary, we recover the 
classical Casselman-Shalika formula for the spherical Whittaker function.}
The proofs are based on localization in equivariant K-theory, and require a geometric interpretation of how the Hecke  inverse
of a D-L operator acts on the class of a point. We prove that the Hecke inverse operators give Grothendieck-Serre dual classes of the motivic classes, a result which might be of independent interest. In an Appendix joint with Dave Anderson we show that if the line bundle is trivial, we recover a generalization of a classical formula by Kostant, Macdonald, Shapiro and Steinberg for the Poincar{\'e} polynomial of $G/B$; the generalization we consider is due to Aky{\i}ld{\i}z and Carrell and replaces $G/B$ by any smooth Schubert variety.        
\end{abstract} 

\maketitle
\section{Introduction} Let $G$ be a complex {reductive group}. Among the most influential formulas in 
mathematics is the Weyl character formula for an irreducible representation of $G$. For the purpose of 
this paper we consider the 
following version of this formula. Let $T \subset B \subset G$ be a Borel subgroup containing a 
maximal torus. Let $X :=G/B$ be the generalized flag manifold and 
$\mathcal{L}_\lambda:= G \times^B \C_\lambda$ 
the line bundle having fibre of weight $\lambda$ over $1.B$. {If $\lambda$ is 
anti-dominant, then $H^i(X,\calL_\lambda)=0$ for all $i>0$, and 
$H^0(X,\calL_\lambda)\simeq V_{w_0\lambda}$, 
where $w_0$ is the longest element in the Weyl group and $V_{w_0\lambda}$ is the 
irreducible highest weight representation with highest weight 
$w_0\lambda$; denote by $\chi_{w_0\lambda}$ its character.
The Weyl character formula calculates 
$\chi_{w_0\lambda}$ as the 
equivariant sheaf Euler characteristic:
\begin{equation}\label{E:euler} 
\chi_{w_0\lambda}=\chi(X, \mathcal{L}_\lambda) := \sum_i (-1)^i H^i(X, \mathcal{L}_\lambda)=\sum_{w\in W} \frac{e^{w\lambda}}{\prod_{\alpha>0}(1-e^{w\alpha})} \/,
\end{equation}
where the $T$-module 
$H^i(X, \mathcal{L}_\lambda)$ is identified with its character.
There are many proofs of this formula, and some of the earliest involved variants of the localization in $K_T(X)$, the $T$-equivariant K-theory of $X$; see e.g.~\cite{demazure:sur,demazure:desingularisations,nielsen:diag}. The goal of this paper is to study a generalization of the Euler characteristic in \eqref{E:euler} in the case when $\mathcal{L}_\lambda$ is replaced by $MC_y(X(w)^\circ) \otimes \mathcal{L}_\lambda$, the multiplication of $\mathcal{L}_\lambda$ by the motivic Chern class of a Schubert cell $X(w)^\circ \subset X$. We do not impose any restrictions on the character $\lambda$, and we regard 
$\chi(X, \mathcal{L}_\lambda)$
as (the character of) a virtual representation.
 
The motivic Chern transformation, defined in \cite{brasselet.schurmann.yokura:hirzebruch}, and extended to the equivariant case in \cite{FRW:motivic,AMSS19:motivic}, is a group homomorphism $MC_y: G_0^T(var/X) \to K_T(X)[y]$ from 
the Grothendieck group of equivariant varieties and morphisms over $X$, modulo the scissor relations, to the equivariant K-theory ring of $X$, to which one adjoins a formal parameter $y$. The transformation is determined by a functoriality property, and the normalization \[ MC_y[id_X: X \to X]= \lambda_y(T^*_X) := \sum y^i [\wedge^i T^*_X] \/, \] in the case when $X$ is smooth; see \S \ref{sec:motivic}. The right hand side is called the (Hirzebruch) $\lambda_y$-class of $X$, and it was used to study Riemann-Roch type statements; cf.~ e.g. ~\cite{hirzebruch:topological}. 

{In \cite{AMSS19:motivic}, a specialization of the equivariant K-theory $K_T(G/B)[y,y^{-1}]$ was identified, as a 
Hecke module, to the Iwahori-fixed part of the the principal series representation of a $p$-adic Langlands dual group. Under this identification, certain `Segre motivic classes' closely related to motivic Chern classes of Schubert cells, were sent to the standard basis elements of the principal series representation, and the `fixed point classes' were sent to the Casselman basis elements. In this
paper we make one step further.}   
It turns out that a variant of the equivariant Euler characteristic $\chi(X; MC_y(X(w)^\circ) \otimes \mathcal{L}_\lambda)$ 
{matches} classical and more recent formulas for the Iwahori-Whittaker functions associated to the principal series representation
\cite{reeder:padic,brubaker.bump.licata,lee.lenart.liu:whittaker,brubaker2019colored}, {see \S \ref{sec:whittaker}}. 
Our goal is to provide a geometric {interpretation} 
of the formulas from {\em loc.cit}, along with efficient proofs, based on equivariant localization. In fact, 
interpretations ({\em a posteriori}) closely related to ours were already present in both Reeder's and Brubaker, Bump and Licata's papers; we believe that our treatment based on the theory of motivic Chern classes simplifies, and provides a natural context, for these results.

We now give a brief description of our main results. Let $W$ be the Weyl group of $(G,T)$, and for $w \in W$ let $X(w)^\circ:= BwB/B$ denote the Schubert cell. Its motivic Chern class is $MC_y(X(w)^\circ) := MC_y[X(w)^\circ \hookrightarrow X]$, i.e. the class associated to the inclusion of the cell in $X$. For each simple root $\alpha_i$, let $\partial_i:K_T(X) \to K_T(X)$ denote the usual Demazure operator {(see \S \ref{sec:DL})}, and consider two Demazure-Lusztig operators acting on $K_T(X)[y]$:
\[ \mathcal{T}_i: = (1+y \mathcal{L}_{\alpha_i}) \partial_i - id; \quad \mathcal{T}_i^\vee: = \partial_i (1+y \mathcal{L}_{\alpha_i}) - id \/. \] 
Both operators satisfy the braid relations and the usual  quadratic relations in the Hecke algebra of $W$; cf.~\cite{lusztig:eqK}, see also Proposition \ref{prop:hecke-relations} below. They are adjoint to each other with respect to the K-theoretic Poincar{\'e} pairing. Localization at the point $1.B \in X$ induces an algebraic version of these operators, defined on $K_T(pt)[y]=R(T)[y]$, where $R(T)$ denotes the representation ring of $T$:
\[ \widetilde{\partial}_i (e^\lambda) := \partial_i(\mathcal{L}_\lambda)_{|1.B} \/; \quad \widetilde{T_i}:=(1+ye^{\alpha_i})\widetilde{\partial_i}-1,\textit{\quad and \quad }\widetilde{T^\vee_i}:=\widetilde{\partial_i}(1+ye^{\alpha_i})-1\/. \]
See \S \ref{ss:DLops} below.
Since the braid relations hold, both types of operators can be defined for any element $w \in W$. 
Our main result is the following (cf.~Theorem \ref{thm:characters} below).

\begin{thm}\label{thm:main1} Let $e^\lambda$ be any character of $T$ and let $\mathcal{L}_\lambda :=G \times^B \C_\lambda$ be the associated line bundle. Then the following hold:

(a) $
\chi\left(X, \calL_\lambda\otimes MC_y(X(w)^\circ)\right)=\widetilde{T^\vee_{w}}(e^\lambda)$;

(b) Let $\lambda_y(id):= \prod_{\alpha >0} (1+ye^\alpha)$ denote the equivariant $\lambda_y$ class of the cotangent space $T^*_{id}X$, and
denote by \[ MC'_y(X(w)^\circ):= \lambda_y(id) \frac{MC_y(X(w)^\circ)}{\lambda_y(T^*_X)} \/; \] this is an element in an appropriate localization of $K_T(X)[y]$. Then 
\[ \chi\left(X, \calL_\lambda\otimes MC_y'(X(w)^\circ)\right)=\widetilde{T_{w}}(e^\lambda) \/. \]
\end{thm}
{The expressions on the right hand side of this theorem have interpretations in $p$-adic representation theory.} 
For instance, the element $\widetilde{T_{w}}(e^\lambda)$ from 
(b) is the Iwahori-Whittaker function from \cite[Theorem 1]{brubaker.bump.licata}, {see Proposition \ref{prop:bbl}(3) below}. By the normalization property of motivic classes, $\sum_{w \in W} MC_y(X(w)^\circ) = \lambda_y(T^*_X)$, thus summing over all $w \in W$ in part (b), and by the Weyl character formula \eqref{E:euler}, one obtains 
\[
\lambda_y(id) \cdot\chi(X, \mathcal{L}_\lambda) = \sum_w \widetilde{T_w}(e^\lambda) \/. 
\]
This yields the classical Casselman-Shalika formula for the spherical Whittaker function \cite[Theorem 5.4]{casselman1980unramifiedII}, see Proposition \ref{prop:bbl} and Corollary \ref{cor:geomwhit} below. 
{A similar approach to the Casselman-Shalika formula, using the equivariant K-theory of the cotangent bundle $T^*_X$, was obtained in \cite[\S 9.6]{su2017k}.} The same normalization property applied to part (a) implies that \[ \sum_{w \in W} \widetilde{T_w^\vee} (e^\lambda) = \chi(X, \mathcal{L}_\lambda \otimes \lambda_y(T^*_X)) \/.\] This interpretation appeared in Reeder's paper \cite{reeder:padic}, and he proves it using localization techniques - the Lefschetz fixed point formula for a certain left multiplication morphism on $X$. (The use of this formula in a similar context goes back at least to Macdonald \cite{macdonald:poincare}.) In fact, Reeder obtained {\em two} cohomological interpretations, the other one in terms of certain local cohomology groups and the Cousin-Grothendieck complex; it would be very interesting to see if similar local cohomology groups {may be utilized to calculate $\chi(X; \mathcal{L}_\lambda \otimes MC_y(X(w)^\circ))$}.

The proof of part (a) is based on the fact proved in \cite{AMSS19:motivic} that $MC_y(X(w)^\circ) = \mathcal{T}_{w^{-1}}(\iota_{id})$, where $\iota_{id} \in K_T(X)$ is the class of the (equivariant) structure sheaf of the point $1.B$. Once this formula is available, the proof is extremely fast: the left hand side is given by the K-theoretic Poincar{\'e} pairing $\langle \mathcal{L}_\lambda, \mathcal{T}_{w^{-1}}(\iota_{id}) \rangle$, which by adjointness of $\mathcal{T}_{w^{-1}}$ and $\mathcal{T}_w^\vee$ equals to \[ \langle \mathcal{T}_w^\vee (\mathcal{L}_\lambda), \iota_{id} \rangle = \widetilde{T^\vee_{w}}(e^\lambda) \/. \] We refer to the proof of Theorem \ref{thm:characters} for the full details. The proof of part (b) is completely similar, but it requires the identification of the class corresponding to the operator $\mathcal{T}_w^\vee$. It was proved in \cite{AMSS19:motivic} that the {\em inverse} of this operator gives an analogue of the dual Segre class for motivic Chern classes; cf.~\eqref{equ:mot} below.  Therefore this problem requires us to consider the standard involution in the Hecke algebras $\mathcal{T}_w \mapsto \overline{\mathcal{T}_w}:= \mathcal{T}_{w^{-1}}^{-1}$ and analyze how the resulting operators are related to the motivic Chern classes. 

The answer turns out to be related to the Grothedieck-Serre duality operator. Recall that this operator is defined by $\calD(-):=\RHom_{X}(-,\omega_X^\bullet)$, where \[\omega_X^\bullet=\omega_{X}[\dim X]=(-1)^{\dim X}\mathcal{L}_{2 \rho}\]  is the shifted canonical bundle and $\rho$ is half the sum of positive roots. We extend it to $K_T(X)[y,y^{-1}]$ by requiring that $\calD(y^i)=y^{-i}$. We prove in Proposition \ref{prop:DT} that \[ \overline{\mathcal{T}_w}= \mathcal{D} \circ \mathcal{T}_w \circ \mathcal{D}\] as operators in $K_T(X)[y,y^{-1}]$. Then the following theorem answers our question above and it may be of independent interest; see Theorem \ref{thm:Doperator} below. To state it, let $Y(w)^\circ= B^-wB/B$ be the opposite Schubert cell, defined with respect to the Borel group $B^-=w_0 B w_0$; let $\iota_{w_0}$ be the class of the structure sheaf of the point $1.B^-$; and let $\lambda_y(w_0)$ be the equivariant $\lambda_y$-class of $T^*_{w_0}X$, the cotangent space at $1.B^-$. 
\begin{thm}\label{thm:main2} For any $w \in W$,
\[ \calD(MC_y(X(w)^\circ)) =\overline{\mathcal{T}_{w^{-1}}}(\iota_{id}); \quad \frac{MC_y(Y(w)^\circ)}{\lambda_y(T^*_X)} = \frac{1}{\lambda_y(w_0)} \mathcal{T}_{(w_0 w)^{-1}}^\vee(\iota_{w_0}) \/. \]
In particular, 
\[ \calD(MC_y(Y(w)^\circ)) = \overline{\mathcal{T}_{(w_0 w)^{-1}}} (\iota_{w_0}); \quad \frac{MC_y(X(w)^\circ)}{\lambda_y(T^*_X)} = \frac{1}{\lambda_y(id)} \mathcal{T}_{w^{-1}}^\vee(\iota_{id}) \/. \]
\end{thm}

Another application of the localization formula for motivic Chern classes is obtained in an Appendix joint with Dave Anderson. If $\mathcal{L}_\lambda$ is trivial, then $\chi(X; MC_y(X)) = \chi(X;\lambda_y(T^*_X))$ can be calculated in two ways: on one side it yields the Poincar{\'e} polynomial of $X$; on the other side one can use the localization theorem.  The result is a formula for the Poincar{\'e} polynomial of $X$:

\begin{thm}\label{thm:mainapp} Let $w \in W$ be a Weyl group element such that the Schubert variety $X(w) \subset G/B$ is smooth. Then the Poincar{\'e} polynomial of $X(w)$ satisfies the equality 
\begin{equation*}\label{equ:intro-identity}
P(X(w),q)=\sum_{v \le w}q^{\ell(v)}=\prod_{\alpha>0,s_\alpha\leq w}\frac{1-q^{\h(\alpha)+1}}{1-q^{\h(\alpha)}},
\end{equation*} where $\ell(w)$ denotes the length of $w$, and $\h(\alpha)$ is the height of the root $\alpha$. 
\end{thm} 
See Corollary \ref{cor:app} below. 
This version of the formula was proved by Aky{\i}ld{\i}z and Carrell \cite{akyildiz.carrell:Betti}, and it has a distinguished history, going back to Kostant \cite{kostant1959principal}, Macdonald \cite{macdonald:poincare}, Shapiro and Steinberg \cite{steinberg1959finite}, who proved it for $X(w) = X$.  In fact, we deduce it from a slightly more general statement about projective manifolds with a very special $\mathbb{G}_m$ action (Theorem~\ref{thm:app}).

{\em Acknowledgments:} The authors thank Paolo Aluffi and J{\"o}rg Sch{\"u}rmann for useful comments on this draft, and for collaborations on several related projects. This project started when the authors attended one of the workshops of the program ``Quiver Varieties and Representation Theory" at CRM in Montreal; both authors would like to thank the organizers for the stimulating environment, and CS would like to thank for the travel support. {We thank the anonymous referees for careful reading and useful suggestions. }

\section{Preliminaries} In this section, we recall basic facts about the equivariant K-theory ring, focusing eventually to the {special} case of the flag varieties. Our main references are \cite{chriss2009representation,nielsen:diag,lusztig:eqK}. 

\subsection{Equivariant K-theory and equivariant localization} Let $X$ be a projective complex manifold with an action of a {complex reductive} group $H$. The equivariant K-theory ring $K_H(X)$ is the Grothendieck ring generated by symbols $[E]$, where $E \to X$ is an $H$-equivariant vector bundle, modulo the relations $[E]=[E_1]+[E_2]$ for any short exact sequence $0 \to E_1 \to E \to E_2 \to 0$ of equivariant vector bundles. The additive ring structure is given by direct sum, and the multiplication is given by tensor products of vector bundles. Since $X$ is smooth, the ring $K_H(X)$ coincides with the Grothendieck group of {$H$}-linearized coherent sheaves on $X$. Indeed, any {$H$-linearized} coherent sheaf has a finite resolution by $H$-equivariant vector bundles; cf.~\cite[\S 5.1]{chriss2009representation}. 

The ring $K_H(X)$ is an algebra over $K_H(pt) = R(H)$, the representation ring of $H$. Later we will consider two situations: when $H:=G$ is a complex reductive
group, and when $H:=T$ is a maximal torus in $G$. In the second situation the representation ring 
is the Laurent polynomial ring $K_T(pt) =\Z[e^{\pm t_1}, \ldots , e^{\pm t_r}]$ where 
$e^{t_i}$ are characters corresponding to a basis of the Lie algebra of $T$.
This ring admits an action of the Weyl group $W:=N_G(T)/T$, and then $K_G(pt) = K_T(pt)^W$ is the Weyl group-invariant part.
An introduction to equivariant K-theory, and more details, can be found in \cite{chriss2009representation}.

Since $X$ is proper, the push-forward to a point equals the Euler characteristic, or, equivalently, the virtual representation, \[\chi(X, \mathcal{F})= \int_X [\mathcal{F}] := \sum_i (-1)^i H^i(X, \mathcal{F})\in K_H(pt) \/. \] 
For $E,F$ equivariant vector bundles, this gives the K-theoretic Poincar\'e pairing: \[ \langle - , - \rangle :K_H(X) \otimes K_H(X) \to K_H(pt); \quad \langle [E], [F] \rangle := \int_X E \otimes F = \chi(X, E \otimes F) \/. \]  
We recall next a version of the localization theorem in the case when $H:=T$ is a complex torus, which will be used throughout this note. Our main reference is Nielsen's paper \cite{nielsen:diag}; see also \cite{chriss2009representation}. Let $V$ be a (complex) vector space with a $T$-action, and consider its weight decomposition $V = \oplus_i V_{\mu_i}$, where $\mu_i$ is a certain set of weights in the dual of the Lie algebra of $T$. 
The {\em character} of $V$ is the element $ch(V):=\sum_i \dim V_{\mu_i} e^{\mu_i}$, regarded  in $K_T(pt)$. If $y$ is an indeterminate, the $\lambda_y$ class of $V$, denoted $\lambda_y(V)$, is the element 
\[ \lambda_y(V) =\sum_{i \ge 0} y^i ch(\wedge^i V) \in K_T(pt)[y] \/.\] 
A standard argument shows that for a short exact sequence 
$0 \to V_1 \to V_2 \to V_3 \to 0$, the $\lambda_y$ class is multiplicative, i.e. $\lambda_y(V_2) = \lambda_y(V_1) \lambda_y(V_3)$. 
In particular, $\lambda_y(V) = \prod_i (1+y e^{\mu_i})^{\dim V_{\mu_i}}$. 
Let $S$ be the multiplicative subset of $K_T(pt)$ generated by elements of 
the form $1-e^\mu$ for nontrivial torus weights $\mu$. 
If the fixed locus $V^T = \{ 0 \}$, then $S$ contains the element $\lambda_{-1}(V) = \prod (1-e^{\mu_i})^{\dim V_{\mu_i}}$.

Let now $X$ be a smooth, projective variety with a $T$-action, and assume that the fixed point set $X^T$ is finite. For each fixed point $x \in X^T$, denote by $T_x X$ the tangent space. This is a vector space with a $T$-action induced from that on $X$. Any $T$-linearized coherent sheaf $\mathcal{F}$ on $X$ determines a class $[\mathcal{F}] \in K_T(X)$. If $\mathcal{F}=E$ is a locally free sheaf, its $\lambda_y$ class is $\lambda_y(E)= \sum y^i [\wedge^i E] \in K_T(X)[y]$; as for vector spaces, this class is multiplicative for short exact sequences. For each $x \in X^T$, let $i_x: \{ x \} \to X$ denote the inclusion. This is a $T$-equivariant proper morphism, and it induces a map $i_x^*: K_T(X) \to K_T(\{x\}) = K_T(pt)$. Denote by $K_T(X)_{loc}$ respectively $K_T(pt)_{loc}$ the localization of $K_T(X)$ and of $K_T(pt)$ at $S$. Since $R(T)$ is a domain, we may identify $K_T(X)$ with a subring inside its localization.   We need the following simplified version of the localization theorem; cf.~ \cite{nielsen:diag}.

\begin{thm}\label{thm:loc} Let $N$ be the normal bundle of $X^T$ in $X$. Then the following hold:

(a) The class $\lambda_{-1}(N^\vee)$ is a unit in $K_T(X)_{loc}$;

(b) When $x$ varies in $X^T$, the structure sheaves $\iota_x:=[\cO_x]$ of the fixed points form a $K_T(pt)_{loc}$-basis of $K_T(X)_{loc}$;

(c) For any $T$-linearized coherent sheaf $\mathcal{F}$ on $X$, the following formula holds:

\[ \chi(X, \mathcal{F}) = \sum_{x \in X^T} \frac{\iota_x^* [\mathcal{F}]}{\lambda_{-1}(T^*_x X)}  \in K_T(pt)_{loc} \/. \]

\end{thm}

This theorem implies in particular that if $E \to X$ is a $T$-equivariant vector bundle over $X$ then \begin{equation}\label{E:locy} \chi(X, \lambda_y(E)) = \sum_{x \in X^T} \frac{\lambda_y(E_x)}{\lambda_{-1}(T^*_x X)} \in K_T(pt)_{loc}[y]\/. \end{equation}

\subsection{Flag varieties}\label{sec:flagv} Let $G$ be a complex {reductive} group with a Borel subgroup $B$ and a maximal torus $T\subset B$. Denote by $B^-$ the opposite Borel subgroup. As before, let $W:=N_G(T)/T$ be the Weyl group, and $\ell:W \to \mathbb{N}$ be the associated length function. Denote by $w_0$ the longest element in $W$; then $B^- = w_0 B w_0$. 
Let also $\Delta := \{ \alpha_1, \ldots , \alpha_r \} \subset R^+$ denote the set of simple roots included in the set of positive roots for $(G,B)$. Let $\rho$ denote the half sum of the positive root. The simple reflection for the root $\alpha_i \in \Delta$ is denoted by $s_i$ and the {corresponding} {\em minimal} parabolic subgroup is denoted by $P_i$, containing the Borel subgroup $B$. 

Let $X:=G/B$ be the flag variety. The group $G$ acts by left multiplication, and $X$ is homogeneous under this action. It has a stratification by Schubert cells $X(w)^\circ:= BwB/B$, respectively opposite Schubert cells $Y(w)^\circ:=B^- w B/B$. The closures $X(w):= \overline{X(w)^\circ}$ and $Y(w):=\overline{Y(w)^\circ}$ are the Schubert varieties. With these definitions, $\dim_{\C} X(w) = \mathrm{codim}_{\C} Y(w) = \ell(w)$. The Weyl group $W$ admits a partial ordering, called the Bruhat ordering, defined by $u \le v$ if and only if $X(u) \subset X(v)$. For any character $\lambda$ of the maximal torus, define $\calL_\lambda:=G\times^B\bbC_\lambda$; this is a $G$-equivariant line bundle on $X$. 

Let $\cO_w:=[\cO_{X(w)}]$ be the Grothendieck class determined by the structure sheaf of $X(w)$ (a coherent sheaf), and similarly $\cO^{w}:= [\cO_{Y(w)}]$; both of these are classes in $K_T(X)$. The equivariant $K$-theory ring $K_T(X)$ has $K_T(pt)$-bases $\{ \cO_w \}_{w \in W}$ and $\{ \cO^{w} \}_{w \in W}$.~Consider next the localized equivariant K-theory ring $K_T(G/B)_{\textit{loc}}$. 
The Weyl group elements $w\in W$ are in bijection with the torus fixed points $e_w:=wB \in G/B$. Let $\iota_w:=[\cO_{e_w}] \in K_T(G/B)$ be the class of the structure sheaf of $e_w$. By the localization Theorem \ref{thm:loc}, the classes $\iota_w$ form a basis for the localized equivariant K-theory ring $K_T(X)_{loc}$; we call this the {\em fixed point basis}. Let $i_w: \{ e_w \} \to X$ be the inclusion; then for any $[\calF]\in K_T(G/B)$, let $[\calF]_{|w}\in K_T(pt)$ denote the pullback $i_w^*[\calF]$ to the fixed point $e_w$.

If $n_w$ is a representative for a Weyl group element $w$, the left multiplication by $n_w$ induces an automorphism $\Phi_w: G/B \to G/B$. This automorphism is not $T$-equivariant, but it is equivariant with respect to the conjugation map {$\varphi_w: T \to T$ defined by $t \mapsto n_w t n_w^{-1}$. In other words, 
\[ \Phi_w(t.gB) = \varphi_w(t).\Phi_w(gB) \/;\]}see e.g.~\cite{knutson:noncomplex,mihalcea2020left} and \cite[\S 5.2]{AMSS:shadows}. Therefore the pull-back of $\Phi_w$ induces an automorphism $\Phi_w^*: K_T(G/B) \to K_T(G/B)$ which twists the representation ring $K_T(pt)$ according to $\varphi_w^*$. The map $\varphi_w^*$ sends {the character $e^\lambda: T \to \C^*$ to the character $e^{w^{-1}(\lambda)}$ defined by $t \mapsto e^\lambda(n_w t n_w^{-1})$.} 
This gives a {(contragredient)} left action of $W$ on $K_T(pt)$ defined by 
$w.e^\lambda:=e^{w{(\lambda)}}$.
In particular, the automorphism $\Phi_w^*$ {\em fixes} the $W$-invariant part $K_T(pt)^W= K_G(pt)$. The pull-back by $\Phi_w$ gives a left action of $W$ on $K_T(X)$ defined by
$w.[\mathcal{F}]:= \Phi_{w^{-1}}^* ([\mathcal{F}])$. In terms of localizations:
\[(w.[\calF])_{|u}=w([\calF]_{|{w^{-1}u}})\/;\]
see \cite{knutson:noncomplex,mihalcea2020left}
and also \cite{tymoczko:permutations,IMN:double} for more on Weyl group actions on equivariant cohomology.
For any $u,w \in W$, this action satisfies:
\[ w.(\iota_{u} )= \iota_{w u} \/; \quad w_0.(\cO_w) = \cO^{w_0w} \/.\] 
The first equality follows because $\Phi_w^{-1} (e_{u}) = e_{w^{-1}u}$, and the second equality because $\Phi_{w_0}^{-1} (X(w)) = \overline{w_0 B w B/B}= Y(w_0 w)$ (recall that $w_0^{-1} = w_0$).

\subsection{Demazure-Lusztig operators}\label{sec:DL} Fix a simple root $\alpha_i \in \Delta$ and $P_i \subset G$ the corresponding minimal parabolic group. Consider the fiber product diagram: 
\begin{equation}\label{E:fibrediag} 
\xymatrix@C=50pt{
FP:= G/B \times_{G/P_{i}}
G/B \ar[r]^-{pr_1}\ar[d]^{pr_2} & G/B \ar[d]^{p_{i}} \\ 
G/B \ar[r]^{p_{i}} & G/P_{i}
} 
\end{equation} 
The Demazure operator \cite{demazure:desingularisations} $\partial_i: K_T(X) \to K_T(X)$ is defined by $\partial_i:= (p_i)^* (p_i)_*$. It satisfies 
\begin{equation}\label{equ:BGGonstru}
 \partial_i(\cO_w) = \begin{cases} \cO_{ws_i} & \textrm{ if } ws_i>w \/; \\ \cO_w & \textrm{ otherwise } \/. \end{cases} 
\end{equation}
See e.g. \cite{kostant1990t}. From this, one deduces that $\partial_i^2 = \partial_i$ and the operators $\partial_i$ satisfy the same braid relations as the elements in the Weyl group $W$. 

We define next the main operators used in this paper. First, consider the projection $p_i:G/B \to G/P_i$ determined by the minimal parabolic subgroup $P_i$, and let $T^*_{p_i}$ be the relative cotangent bundle. It is given by \[ T^*_{p_i}  = \mathcal{L}_{\alpha_i} = G \times^B \C_{\alpha_i} \/, \] i.e. it is the equivariant line bundle on $G/B$ with character $\alpha_i$ in the fibre over $1.B$. 

\begin{defin}\label{def:hecke} Let $\alpha_i \in \Delta$ be a simple root. Define the operators on $K_T(X)[y]$ \[ \mathcal{T}_i: = \lambda_y(T^*_{p_i}) \partial_i - id; \quad \mathcal{T}_i^\vee: = \partial_i \lambda_y(T^*_{p_i}) - id \/. \] \end{defin}
These two operators are $K_T(pt)[y]$-module endomorphisms of $K_T(X)[y]$, and they are adjoint to each other, i.e. $\langle \calT_i (a), b\rangle = \langle a, \mathcal{T}_i^\vee (b) \rangle$; cf.~\cite[Lemma 3.3]{AMSS19:motivic}.

\begin{rem}\label{rmk:convos} The operator ${\mathcal{T}_i}^\vee$ was defined by Lusztig \cite[Equation (4.2)]{lusztig:eqK} in relation to affine Hecke algebras and equivariant K-theory of flag varieties. The `dual' operator $\mathcal{T}_i$ arises naturally in the study of motivic Chern classes of Schubert cells; cf.~\S \ref{sec:motivic} below.
\end{rem}
\begin{prop}[\cite{lusztig:eqK}]\label{prop:hecke-relations} The operators $\mathcal{T}_i$ and $\mathcal{T}_i^\vee$ satisfy the braid relations and the following quadratic relation
\[ (\mathcal{T}_i + id) (\mathcal{T}_i + y) = ({\mathcal{T}_i}^\vee + id) ({\mathcal{T}_i}^\vee + y) = 0 \/. \] 
\end{prop}
 
The action of the Hecke operators on the fixed point basis is recorded below; see e.g.~\cite[Lemma 3.7]{AMSS19:motivic} for a proof.
\begin{lemma}\label{lem:actiononfixedpoint} The following formulas hold in $K_T(G/B)_{loc}$:

(a) For any weight $\lambda$, $\mathcal{L}_\lambda \cdot \iota_w = e^{w \lambda} \iota_w$, and $\chi(G/B,\iota_w)=1$;

(b) For any simple root $\alpha_i$, \[ \partial_i (\iota_w) =\frac{1}{1-e^{w\alpha_i}}\iota_{w}+\frac{1}{1-e^{-w\alpha_i}}\iota_{ws_{i}} \/; \]

(c) The action of the operator $\calT_i$ on the fixed point basis is given by the following formula
\[\calT_i(\iota_{w})=-\frac{1+y}{1-e^{-w\alpha_i}}\iota_{w}+\frac{1+ye^{-w\alpha_i}}{1-e^{-w\alpha_i}}\iota_{ws_{i}}.\]

(d) The action of the adjoint operator $\calT_i^\vee$ is given by
\[\calT^\vee_i(\iota_{w})=-\frac{1+y}{1-e^{-w\alpha_i}}\iota_{w}+\frac{1+ye^{w\alpha_i}}{1-e^{-w\alpha_i}}\iota_{ws_{i}} \/.\]
\end{lemma}
\subsection{Demazure-Lusztig operators on $K_T(pt)$}\label{ss:DLops}
Define $c: K_T(pt) \to K_G(G/B)$ by $e^\lambda \mapsto [\mathcal{L}_\lambda]$. This is an isomorphism of $K_G(pt)$-algebras, with inverse given by {$i^*_{id} [\mathcal{F}] = [\mathcal{F}]_{|id}$}, the localization at $1.B$; see e.g.~\cite[Chapter 6]{chriss2009representation}. Now let ${\calA}: K_G(G/B) \to K_G(G/B)$ be any $K_G(pt)$-linear endomorphism.~Associated to ${\calA}$ one can define the linear map $\widetilde{A}:K_T(pt) \to K_T(pt)$ given by composition \[ \widetilde{A}:\xymatrix{ K_T(pt) \ar[r]^{c} & K_G(G/B) \ar[r]^{{\calA}} & K_G(G/B) \ar[r]^{i^*_{id}} & K_T(pt) \/.} \] 
The homomorphism $\widetilde{A}$ is $K_G(pt)$-linear. Equivalently, 
\begin{equation}\label{equ:tildeoperators}
\widetilde{A}(e^\lambda) = \langle {\calA} (\mathcal{L}_\lambda), \iota_{id} \rangle= {\calA} (\mathcal{L}_\lambda)_{|id}\/. 
\end{equation} 
The last equality holds because by projection formula $\langle {\calA} (\mathcal{L}_\lambda), \iota_{id} \rangle = \int_{e_{id}} i_{id}^*({\calA} (\mathcal{L}_\lambda)) =  i_{id}^*({\calA} (\mathcal{L}_\lambda)) $. We record the formulas for the associated Demazure-Lusztig operators from the previous section.
\begin{cor}\label{cor:operators} 
The following hold in $\End_{\bbZ[y]}K_T(pt)[y]$:
\[\widetilde{\partial_i}(e^\lambda)=\frac{e^\lambda}{1-e^{\alpha_i}}+\frac{e^{s_i\lambda}}{1-e^{-\alpha_i}},\] 
\[\widetilde{T_i}(e^\lambda)=-e^\lambda\frac{1+y}{1-e^{-\alpha_i}}+e^{s_i\lambda}\frac{1+ye^{\alpha_i}}{1-e^{-\alpha_i}},\] 
\[\widetilde{T_i^\vee}(e^\lambda)=-e^\lambda\frac{1+y}{1-e^{-\alpha_i}}+e^{s_i\lambda}\frac{1+ye^{-\alpha_i}}{1-e^{-\alpha_i}}.\]
\end{cor} 
\begin{rem} The operator $\widetilde{T_i}$ appeared in \cite[Equation (3)]{brubaker.bump.licata} and \cite[\S 2]{lee.lenart.liu:whittaker}, in relation to Whittaker functions; the operator $\widetilde{T^\vee_i}$ appeared in \cite[Equation (8.1)]{lusztig:eqK} (with the opposite choice of positive roots).
Also observe that 
\[\widetilde{T_i}=(1+ye^{\alpha_i})\widetilde{\partial_i}-1 \textit{\quad and \quad }\widetilde{T^\vee_i}=\widetilde{\partial_i}(1+ye^{\alpha_i})-1.\]
\end{rem}
\begin{proof}[Proof of Corollary \ref{cor:operators}] By definition and the self-adjointness of $\partial_i$,
\[ \widetilde{\partial_i}(e^\lambda) = \langle \partial_i(\mathcal{L}_\lambda),\iota_{id} \rangle = \langle \mathcal{L}_\lambda, \partial_i(\iota_{id}) \rangle\/; \] then the result follows from (a) and (b) of the Lemma \ref{lem:actiononfixedpoint}.

Since $\calT_i$ and $\calT^\vee_i$ are adjoint to each other, we obtain
\[\widetilde{T_i}(e^\lambda)=\langle\calT_i(\calL_\lambda),\iota_{id}\rangle=\langle\calL_\lambda,\calT_i^\vee(\iota_{id})\rangle=-e^\lambda\frac{1+y}{1-e^{-\alpha_i}}+e^{s_i\lambda}\frac{1+ye^{\alpha_i}}{1-e^{-\alpha_i}}.\]
The last equality follows similarly.
\end{proof}

\section{Serre duality and the Hecke involution} As before, $X=G/B$. In this section we compare two duality operators. The first is the Grothendieck-Serre duality operator $\calD(-):=\RHom_{X}(-,\omega_X^\bullet)$, where \[\omega_X^\bullet=\omega_{X}[\dim X]=(-1)^{\dim X}\mathcal{L}_{2 \rho}\]  is the shifted canonical bundle. We extend it to $K_T(X)[y,y^{-1}]$ by requiring that $\calD(y^i)=y^{-i}$; then $\mathcal{D}^2 = id$. Since the canonical bundle $\omega_X$ is locally free it follows that for a sheaf $\mathcal{F}$ on $X$, $\calD(\mathcal{F}):=(-1)^{\dim X} [\mathcal{F}]^\vee \otimes [\omega_X]$, where $[\mathcal{F}]^\vee$ is calculated by taking a (finite) resolution by vector bundles, then taking the dual of this resolution. Of course, the same works in the context of equivariant K-theory and equivariant sheaves, using an equivariant resolution by vector bundles \cite{chriss2009representation}. The presence of the vector bundle dual implies that $\mathcal{D}$ is not $K_T(pt)$-linear, but it twists the coefficients in $K_T(pt)$ by sending $e^\lambda \mapsto (e^\lambda)^\vee:= e^{-\lambda}$. Since the pull-back induces a ring homomorphism in equivariant K-theory it follows that for any fixed point $e_w$ in $X$,  \begin{equation}\label{E:Dcompat} i_w^*[\calD(\mathcal{F})] = (-1)^{\dim X} ([\mathcal{F}]_{|w})^\vee \cdot (\omega_X)_{|w} \/. \end{equation} We need the following Lemma:

\begin{lemma}\label{lemma:Diw} For any $w \in W$, $\mathcal{D}(\iota_w) = \iota_w$. \end{lemma} 
\begin{proof} By the equation \eqref{E:Dcompat}, it follows that 
\[\calD(\iota_w)|_u=\delta_{w,u}(-1)^{\dim X}(\iota_w|_w)^\vee \cdot \wedge^{\dim X} (T_{e_w}^* X) =  \delta_{w,u} (-1)^{\dim X} (\lambda_{-1} (T^*_{e_w} X))^\vee \cdot \wedge^{\dim X} (T_{e_w}^* X) \/. \] The weights of the cotangent space $T^*_{e_w}X$ are $e^{w\alpha}$ for $\alpha$ varying over the positive roots. The term on the right equals \[ \delta_{u,w} (-1)^{\dim X} \prod_{\alpha>0} (1-e^{-w(\alpha)}) \times \prod_{\alpha >0}  e^{w(\alpha)} = \delta_{u,w} \prod_{\alpha>0} (1-e^{w(\alpha)}) = (\iota_w)_{|u} \/. \] Combining the two equations shows that $\calD(\iota_w)|_u =  (\iota_w)_{|u}$, from which we deduce the lemma, by injectivity of the localization map. 
\end{proof}

The second duality is determined by an involution on the Hecke algebra. Define: \begin{equation}\label{E:defbarT} \overline{\mathcal{T}_w}:=\mathcal{T}_{w^{-1}}^{-1}; \quad  \overline{\mathcal{T}_w^\vee}:= (\mathcal{T}_{w^{-1}}^\vee)^{-1} \/. \end{equation} We extend the bar operation by requiring $\bar{y}=y^{-1}$. 
The next lemma shows that the Grothendieck--Serre duality $\calD$ corresponds to the Hecke involution.
\begin{prop}\label{prop:DT}
For any $w\in W$, 
\[ \overline{\mathcal{T}_w} = \calD\circ \calT_w\circ \calD \/, \]
as $K_T(pt)[y]$-linear endomorphisms of $K_T(X)[y]$.
\end{prop}
\begin{proof}  Observe that since the vector bundle duality is an involution, the composition $\calD\circ \calT_w\circ \calD$ is $K_T(pt)[y]$-linear. Moreover, since $\calD^2=id$, we only need to show that for any simple root $\alpha_i$,
\begin{equation}\label{equ:DT}
\calT^{-1}_{i}=\calD\circ \calT_{i}\circ \calD \/.
\end{equation} 
To this aim, it suffices to show that the operators agree on the basis of fixed points. 
By Lemma \ref{lemma:Diw}, Lemma \ref{lem:actiononfixedpoint}(c), and since $\calT_{i}^{-1}=-y^{-1}(\calT_{i}+1+y)$, we have 
\begin{align*}
\calD\circ \calT_{i}\circ \calD(\iota_w)&=\calD\circ \calT_{i}(\iota_w)\\
&=\calD\left(-\frac{1+y}{1-e^{-w\alpha_i}}\iota_{w}+\frac{1+ye^{-w\alpha_i}}{1-e^{-w\alpha_i}}\iota_{ws_{i}}\right)\\
&=-\frac{1+y^{-1}}{1-e^{w\alpha_i}}\iota_{w}+\frac{1+y^{-1}e^{w\alpha_i}}{1-e^{w\alpha_i}}\iota_{ws_{i}}\\
&=-y^{-1}(\calT_{i}+1+y)(\iota_w)\\
&=\calT_{i}^{-1}(\iota_w).
\end{align*}
This proves Equation \eqref{equ:DT} and finishes the proof of the lemma.
\end{proof}
\begin{rem} Similar localization techniques can be used to prove that for any $w \in W$, the following equalities of $K_T(pt)[y]$-linear endomorphisms of $K_T(X)[y]$ hold:\[ \calL_{2\rho}\circ \overline{\mathcal{T}_w^\vee} \circ\calL_{-2\rho}=\calD\circ \calT^\vee_w\circ \calD \] 
\[y^{-\ell(w)}\calL_{-\rho}\circ\calT_w\circ\calL_{\rho}= \calT^\vee_w|_{y\rightarrow y^{-1}} \/. \] {These are closely related to the the formulas for conjugating the algebraic Demazure-Lusztig operators $T_w$ as in \cite[\S 8]{brubaker.bump.licata}, in turn related to the study of certain specializations of non-symmetric Macdonald polynomials \cite{ion:nonsymmetric}.} \end{rem}

\section{Motivic Chern classes of Schubert cells and the Hecke operators}\label{sec:motivic}
We recall the definition of the motivic Chern classes, following \cite{brasselet.schurmann.yokura:hirzebruch} and \cite{AMSS19:motivic}. For now let $X$ be a quasi-projective, complex algebraic variety, with an action of a torus $T$. First we recall the definition of the (relative) motivic Grothendieck group ${G}_0^T(var/X)$ of varieties over $X$. This is the free abelian group generated by symbols $[f: Z \to X]$ for isomorphism classes of equivariant morphisms $f:Z \to X$, where $Z$ is a quasi-projective $T$-variety, modulo the usual additivity relations 
$$[f: Z \to X] = [f: U \to X] + [f:Z \setminus U \to X]$$ for $U \subset Z$ an open invariant subvariety.

The following theorem was proved by Brasselet, Sch{\"u}rmann and Yokura \cite[Theorem 2.1]{brasselet.schurmann.yokura:hirzebruch} in the non-equivariant case. Minor changes in arguments are needed in the equivariant case - see \cite{AMSS19:motivic,FRW:motivic}. 

\begin{thm}\label{thm:existence}\cite[Theorem 4.2]{AMSS19:motivic} Let $X$ be a quasi-projective, non-singular, complex algebraic variety with an action of the torus $T$. There exists a unique natural transformation $MC_y: G_0^T(var/X) \to K_T(X)[y]$ satisfying the following properties:
\begin{enumerate} \item[(1)] It is functorial with respect to $T$-equivariant proper morphisms of non-singular, quasi-projective varieties. 

\item[(2)] It satisfies the normalization condition \[ MC_y[id_X: X \to X] = \lambda_y(T^*_X) = \sum y^i [\wedge^i T^*_X]_T \in K_T(X)[y] \/. \]
\end{enumerate}
\end{thm}
If $X$ is understood from the context, and $Y$ is any $T$-invariant subvariety $Y\subset X$, let $MC_y(Y):=MC_y[Y\hookrightarrow X]\in K_T(X)[y]$ denote the motivic Chern class associated to $Y$.

Return to the situation when $X=G/B$ and denote by $\lambda_y(w)$ the $\lambda_y$-class of the cotangent space $T^*_{e_w}(X)$ at the $T$-fixed point $e_w$. We have already seen that $\lambda_y(w):=\prod_{\alpha >0} (1+ y e^{w(\alpha)})$. 
It is proved in  \cite{AMSS19:motivic}, cf.~Theorem. 5.1, Theorem~6.2, and Theorem~8.1, that
\begin{equation}\label{equ:mot}
MC_y(X(w)^\circ)=\calT_{w^{-1}}(\iota_{id}); \quad \frac{\calD(MC_y(Y(w)^\circ))}{\lambda_y(T^*_X)} = \frac{1}{\lambda_y(w_0)} (\mathcal{T}_{w_0w}^\vee)^{-1}(\iota_{w_0}) \/. \end{equation} Further, as a consequence of the Poincar{\'e} duality for the K-theoretic stable envelopes \cite[Proposition 1]{okounkov2016quantum}, it is proved in \cite[Theorem 8.11]{AMSS19:motivic} that 
\begin{equation}\label{E:Poincare} \langle MC_y(X(u)^\circ), \frac{\calD(MC_y(Y(v)^\circ))}{\lambda_y(T^*_X)} \rangle = \delta_{u,v}(-y)^{\ell(v)-\dim X} \/.\end{equation} 
A natural question is what are the classes generated by the operators $\overline{\mathcal{T}}_w$ and $\mathcal{T}_w^\vee$ defined in \eqref{E:defbarT}. The next theorem gives the answer.
\begin{thm}\label{thm:Doperator} For any $w \in W$,
\[ \calD(MC_y(X(w)^\circ)) =\overline{\mathcal{T}_{w^{-1}}}(\iota_{id}); \quad \frac{MC_y(Y(w)^\circ)}{\lambda_y(T^*_X)} = \frac{1}{\lambda_y(w_0)} \mathcal{T}_{(w_0 w)^{-1}}^\vee(\iota_{w_0}) \/. \]
In particular, 
\[ \calD(MC_y(Y(w)^\circ)) = \overline{\mathcal{T}_{(w_0 w)^{-1}}} (\iota_{w_0}); \quad \frac{MC_y(X(w)^\circ)}{\lambda_y(T^*_X)} = \frac{1}{\lambda_y(id)} \mathcal{T}_{w^{-1}}^\vee(\iota_{id}) \/. \]
\end{thm}
The proof will use the following lemma, proved in \cite[Proposition 3.5]{AMSS19:motivic}.
\begin{lemma}\label{lemma:prod} Let $u,v \in W$ be two Weyl group elements. Then \[ \mathcal{T}_u \cdot (\mathcal{T}_v)^{-1} = c_{uv^{-1}}\mathcal{T}_{uv^{-1}} + \sum_{w <uv^{-1}} c_w(y) \mathcal{T}_w \/, \] where $c_w(y)$ is a rational function in $y$. If $\ell(uv^{-1}) = \ell(u) + \ell(v^{-1})$, then $c_{uv^{-1}}(y) = (-y)^{-\ell(v)}$. 
\end{lemma}  
\begin{proof}[Proof of Theorem \ref{thm:Doperator}] The equalities on the second row follow from those on the first row after applying the automorphism obtained by left multiplication by $w_0$ (cf.~\S \ref{sec:flagv} above). To prove the first equality in the first row, observe that by Lemma \ref{lemma:Diw} and Proposition \ref{prop:DT},
\begin{equation}\label{E:pf1}\overline{\calT_{w^{-1}}}(\iota_{id})=\calD\circ \calT_{w^{-1}}\circ \calD(\iota_{id})= \calD(MC_y(X(w)^\circ)).\end{equation}
It remains to prove the second equality. We essentially rewrite the proof of \cite[Theorem 6.2]{AMSS19:motivic}, and for completeness we include the details. Observe that the class $\mathcal{T}_u(\iota_{id})$ is supported on the Schubert variety $X(u^{-1})$, thus its expansion into Schubert classes contains only classes $\cO_v$ where $v \le u^{-1}$. In addition, $\langle \cO_v, \iota_{w_0} \rangle = 0$ for $v<w_0$. Therefore, $\langle \mathcal{T}_u(\iota_{id}), \iota_{w_0} \rangle = 0$ for $u<w_0$. By adjointness,
\[ \langle \overline{\mathcal{T}_{u^{-1}}}(\iota_{id}), \mathcal{T}_{(w_0v)^{-1}}^\vee(\iota_{w_0}) \rangle = \langle  (\mathcal{T}_{w_0v} \cdot  (\mathcal{T}_{u})^{-1})(\iota_{id}), \iota_{w_0} \rangle \/. \]
From Lemma \ref{lemma:prod} the last quantity is nonzero only if $u=v$, and in this case, using for instance 
\cite[Proposition 7.1 (b)]{AMSS19:motivic} we obtain that 
\[ \begin{split} \langle  (\mathcal{T}_{w_0v} \cdot  (\mathcal{T}_{u})^{-1})(\iota_{id}), \iota_{w_0} \rangle  & = (-y)^{-\ell(u)} \langle \mathcal{T}_{w_0vu^{-1}} (\iota_{id}),\iota_{w_0} \rangle  \\ & = (-y)^{-\ell(u)} \langle \mathcal{T}_{w_0vu^{-1}} (\iota_{id}),\iota_{w_0} \rangle \\ & = (-y)^{-\ell(u)} \langle MC_y(X(w_0)^\circ), \iota_{w_0} \rangle \\ & = (-y)^{-\ell(u)}   MC_y(X(w_0)^\circ)_{|w_0} \\ & = (-y)^{-\ell(u)} \lambda_y(w_0) \/.\end{split} \] 
Combining the previous two equalities and the equation \eqref{E:pf1} we obtain that 
\[ \langle \mathcal{D}(MC_y(X(u)^\circ)), \frac{1}{\lambda_y(w_0)} \mathcal{T}_{(w_0 v)^{-1}}^\vee(\iota_{w_0}) \rangle = \delta_{u,v} (-y)^{-\ell(u)} \/. \] On the other side, observe that the left Weyl group multiplication commutes with the Grothendieck-Serre duality $\mathcal{D}$, it fixes both $y$ and $\lambda_y(T^*_X)$ (since $T^*_X$ is a $G$-equivariant bundle), and it satisfies $w_0. MC_y(Y(w)^\circ) = MC_y(X(w_0 w)^\circ)$. Then by the definition of the pairing and by Equation \eqref{E:Poincare} we obtain \[ \begin{split} \langle \mathcal{D}(MC_y(X(u)^\circ)), \frac{MC_y(Y(v)^\circ)}{\lambda_y(T^*_X)} \rangle  & =  \langle MC_y(Y(v)^\circ), \frac{\mathcal{D}(MC_y(X(u)^\circ))}{\lambda_y(T^*_X)} \rangle \\ & = \delta_{u,v} (-y)^{\ell(w_0 u) - \dim X} \\ & = \delta_{u,v} (-y)^{-\ell(u)} \/.\end{split} \] Combining the last two equations, and since the Poincar{\'e} pairing is non-degenerate, it follows that \[ \frac{MC_y(Y(v)^\circ)}{\lambda_y(T^*_X)} = \frac{1}{\lambda_y(w_0)} \mathcal{T}_{(w_0 v)^{-1}}^\vee(\iota_{w_0}) \/, \] as claimed. 
\end{proof}

\section{Character formulas}
In this section we prove our main result, Theorem \ref{thm:characters} below, and we deduce the 
`geometric' versions of the formulas relating the motivic Chern classes to the 
Iwahori-Whittaker functions for the Langlands dual group over a non-Archimedean local field. In \S \ref{sec:whittaker}
we recall the definitions of the latter, and we state the precise dictionary between the two contexts.
This gives a geometric interpretation of some results of Brubaker, Bump and Licata \cite{brubaker.bump.licata}.
  
\subsection{Euler characteristics}\label{sec:char}
For simplicity of notation, denote by \[ MC'_y(X(w)^\circ):=\lambda_y(id) \frac{MC_y(X(w)^\circ)}{\lambda_y(T^*_X)} = \prod_{\alpha >0} (1+ye^\alpha)\frac{MC_y(X(w)^\circ)}{\lambda_y(T^*_X)} \/. \]
Then, by Theorem \ref{thm:Doperator}, $MC_y'(X(w)^\circ) =\mathcal{T}_{w^{-1}}^\vee (\iota_{id})$. We now prove the main result of this note.
\begin{thm}\label{thm:characters}
For any $\lambda\in X^*(T)$ and $w\in W$, the following hold:
\begin{equation*}\label{equ:demazure}
\chi(X, \calL_\lambda\otimes \cO_w)=\widetilde{\partial_w}(e^\lambda),
\end{equation*}
\begin{equation*}\label{equ:whittaker1}
\chi\left(X, \calL_\lambda\otimes MC_y(X(w)^\circ)\right)=\widetilde{T^\vee_{w}}(e^\lambda),
\end{equation*}
and 
\begin{equation*}\label{equ:whittaker}
\chi\left(X, \calL_\lambda\otimes MC_y'(X(w)^\circ)\right)=\widetilde{T_{w}}(e^\lambda).
\end{equation*}
\end{thm}
\begin{rem}
The methods in this paper do not give information about the individual cohomology groups $H^i(X,\mathcal{L}_\lambda \otimes \cO_w)$, so the first equation may be regarded as a virtual Demazure character formula. The usual formula requires that $\lambda$ is antidominant, in which case all higher cohomology groups vanish \cite{demazure:desingularisations,andersen:schubert,mehta.ramanathan}, and $H^0(G/B,\calL_\lambda)$ is the highest weight representation of $G$ with highest weight $w_0\lambda$. Thus, 
 \begin{equation}\label{equ:weyl}
  \chi(G/B,\calL_\lambda)=\widetilde{\partial_{w_0}}(e^\lambda)=\chi_{w_0\lambda},
 \end{equation}
where $\chi_{w_0\lambda}$ is the character of the highest weight representation.  
We regard $\mathcal{L}_\lambda \otimes \mathcal{O}_w$ as a class in $K_T(X)$, rather than restricting $\mathcal{L}_\lambda$ to the Schubert variety $X(w)$. Nevertheless, the K-theoretic push-forward of the inclusion map $i: X(w) \to X$ satisfies $i_*[\cO_{X(w)}] = \cO_w$ (since it is a closed embedding). Thus, $\chi(X(w), i^*\mathcal{L}_\lambda) = \chi(X,\calL_\lambda\otimes {\cO_w})$ by the projection formula. 

Similar considerations apply to the Euler characteristics from the second and third equations.
\end{rem}
\begin{proof}[Proof of Theorem \ref{thm:characters}] The proof relies on the fact that each of the Euler characteristics to be calculated coincide with a Poincar{\'e} pairing $\langle \mathcal{L}_\lambda, \kappa \rangle$, where $\kappa$ is a class given by a version of the Demazure-Lusztig operators. Then one can employ various adjointness properties of these operators to easily conclude the proof.

We start with the first equality. By formula \eqref{equ:BGGonstru}, and since the operators $\partial_i$'s are self adjoint, we have
\begin{align*}
\chi(X,\calL_\lambda\otimes {\cO_w})
&=\langle\calL_\lambda, \calO_{w}\rangle=\langle\calL_\lambda, \partial_{w^{-1}}(\iota_{id})\rangle\\
&=\langle\partial_w(\calL_\lambda), \iota_{id}\rangle\\
&=\widetilde{\partial_w}(e^\lambda),
\end{align*}
where the last equality follows from Equation \eqref{equ:tildeoperators}.

To prove the second equality, recall that the operators $\calT_i$ and $\calT^\vee_i$ are adjoint to each other. Hence,
\begin{align*}
\chi\left(X, \calL_\lambda\otimes MC_y(X(w)^\circ)\right)&=\langle \calL_\lambda, MC_y(X(w)^\circ)\rangle=\langle \calL_\lambda, \calT_{w^{-1}}(\iota_{id})\rangle\\
&=\langle \calT^\vee_w(\calL_\lambda), \iota_{id}\rangle\\
&=\widetilde{T^\vee_w}(e^\lambda),
\end{align*}
where the last equality follows from Equation \eqref{equ:tildeoperators}. The third equality can be proved in the same way, using that $MC'_y(X(w)^\circ) = \mathcal{T}_{w^{-1}}^\vee(\iota_{id})$, by Theorem \ref{thm:Doperator}.
\end{proof}

Summing over all the Weyl group elements, we get the following corollary.
\begin{cor}
For any $\lambda\in X^*(T)$, we have
\[\sum_w \widetilde{T^\vee_w}(e^\lambda)=\sum_w e^{w\lambda}\prod_{\alpha>0}\frac{1+ye^{w\alpha}}{1-e^{w\alpha}},\]
and 
\[\sum_w \widetilde{T_w}(e^\lambda)=\prod_{\alpha>0}(1+ye^\alpha)\partial_{w_0}(e^\lambda)=\prod_{\alpha>0}(1+ye^\alpha)\sum_w\frac{e^{w\lambda}}{\prod_{\alpha>0}(1-e^{w\alpha})}.\]
\end{cor}
\begin{proof}
Since $\sum_wMC_y(X(w)^\circ)=\lambda_y(T^*_X)$, using the localization formula \eqref{E:locy} we obtain
\begin{align*}
\sum_w \widetilde{T^\vee_w}(e^\lambda)&=\sum_w\chi\left(X, \calL_\lambda\otimes MC_y(X(w)^\circ)\right)\\
&=\chi(X, \calL_\lambda\otimes \lambda_y(T^*_X))\\
&=\sum_w e^{w\lambda}\prod_{\alpha>0}\frac{1+ye^{w\alpha}}{1-e^{w\alpha}}.
\end{align*}

Similarly, using $\sum_w MC'_y(X(w)^\circ) = \lambda_y(id)$, we obtain
\begin{align*}
\sum_w \widetilde{T_w}(e^\lambda)&=\sum_w\chi\left(X, \calL_\lambda\otimes MC'_y(X(w)^\circ)\right)\\
&=\prod_{\alpha>0}(1+ye^\alpha)\chi(X, \calL_\lambda)\\
&=\prod_{\alpha>0}(1+ye^\alpha)\sum_w\frac{e^{w\lambda}}{\prod_{\alpha>0}(1-e^{w\alpha})} \/.
\end{align*}
This finishes the proof. 
\end{proof}

\subsection{Iwahori-Whittaker functions}\label{sec:whittaker}
In this section we recall the definition of the Iwahori-Whittaker functions, following 
\cite{brubaker.bump.licata} and {\cite[\S 3]{brubaker2019colored}}. 
We utilize slightly different conventions: 
in the definition of the principal series representation, and 
of the Whittaker functionals, we make the opposite choices for the Borel subgroup, 
and of the unipotent radical. Our main statement is Corollary \ref{cor:geomwhit},
where we relate the Euler characteristics from 
Theorem \ref{thm:characters} above to the Iwahori-Whittaker functions from \cite[Thm.~1]{brubaker.bump.licata}.

Let $F$ be a non-Archimedean local field with ring of integers $\calO$. Let 
$\varpi$ be a uniformizer of $\calO$, and let $\bbF_q$ be the residue field of $\calO$. 
Let $\hat{G}$ be the split reductive 
Langlands dual group of $G$ over $F$. Let $\hat{B}$ and $\hat{T}$ be the 
corresponding dual Borel subgroup and the dual maximal torus, respectively. Let $\hat{N}$ be the unipotent radical of $\hat{B}$, and let $\hat{B}^-$ be the opposite Borel subgroup. 
There is an identification of the
original complex torus $T=X^*(\hat{T})\otimes_{\bbZ} \bbC^*$, where $X^*(\hat{T})$ is the group of characters of $\hat{T}$. This induces a pairing
\[\hat{T}(F)/\hat{T}(\calO)\times T\rightarrow \bbC^*,\quad (t\hat{T}(\calO),\lambda\otimes z)\mapsto z^{\val(\lambda(t))},\]
where $\lambda\in X^*(\hat{T})$, $z\in \bbC^*$ and $\val$ is the valuation of $F$ determined by $\varpi$. 
In turn, the pairing determines a bijection between unramified character of $\hat{T}(F)$ 
(i.e., characters which are trivial on 
$\hat{T}(\calO)$) and the complex torus $T$. For any unramified character $\tau$ of $\hat{T}(F)$, 
the principal series representation of $\hat{G}(F)$ is defined by
\[I(\tau):=\Ind_{\hat{B}^-(F)}^{\hat{G}(F)}\tau=\{\textit{locally constant fuctions }f:
\hat{G}(F)\rightarrow\bbC\mid f(bg)=(\delta^{\frac{1}{2}}\tau)(b)f(g), b\in \hat{B}^-(F),\}\]
where $\delta:\hat{B}^-(F)\rightarrow \bbR^*$ is the modular character, and $\tau$ is extended trivially from $\hat{T}(F)$ to $\hat{B}^-(F)$.
The group $\hat{G}(F)$ acts on $I(\tau)$ from the right; denote this action by $\pi$. 
Then $\pi(g)f(x)=f(xg)$, where $x,g\in \hat{G}(F)$. {Assume in addition that $\tau$ is regular, i.e.~the stabilizer $W_\tau = 1$; in this case $I(\tau)$ is irreducible.}

Let $I\subset \hat{G}(\calO)$ be the Iwahori subgroup, which is the preimage of $\hat{B}(\bbF_q)$ under the reduction map $\hat{G}(\calO)\rightarrow \hat{G}(\bbF_q)$. Utilizing the decomposition
$\hat{G}(F)=\sqcup_{w\in W}\hat{B}^-(F)wI$ one may deduce that
the Iwahori-invariant subspace $I(\tau)^I$ has a basis $\{\Phi_w^\tau \mid w\in W\}$, 
where $\Phi_w^\tau : \hat{G}(F) \to \bbC$ is defined by 
\[\Phi_w^\tau(bw'i)=\begin{cases} (\delta^{\frac{1}{2}}\tau)(b) & \textit{if } w=w',\\
0 & \textit{otherwise,}\end{cases}\]
for any $b\in \hat{B}^-(F)$, $w'\in W$ and $i\in I$. This basis is called the {\em standard basis}. 
{An identification of the Hecke modules 
$K_{T}(G/B)[y,y^{-1}] \otimes_{K_{T}(pt)[y,y^{-1}]} \bbC_\tau \simeq I(\tau)^I $, 
where $y=-q^{-1}$, and sending 
$\lambda_y(w_0) \frac{\calD(MC_{y}(Y(w)^\circ))}{\lambda_y (T^*X)} \otimes 1$ to $\Phi_w^\tau$,
was proved in \cite[Thm. 10.2]{AMSS19:motivic}.}

Finally, let $x_\alpha:\bbG_a\rightarrow \hat{G}$ denote the one-parameter subgroup
determined by a root $\alpha$ of $(G,T)$. Pick any character $\psi$ of $\hat{N}(F)$, 
such that for any positive root $\alpha$, the restriction of $\psi\circ x_\alpha: F\rightarrow \bbC^*$
is trivial on $\calO$, but not on larger fractional ideals. A {\em Whittaker functional} 
is a linear map $\Omega_\tau:I(\tau)\rightarrow \bbC$ such that 
$\Omega_\tau(\pi(n)f)=\psi(n)\Omega_\tau(f)$ for any $n\in \hat{N}(F)$. 
For fixed $\tau$, the vector space of Whittaker functionals is one-dimensional \cite{rodier1972modele}, 
therefore it suffices to consider the following explicit Whittaker functional:
\[\Omega_\tau(f):=\int_{\hat{N}(F)}f(n)\psi(n)^{-1} dn.\]
 
The \textit{Iwahori-Whittaker functions} are obtained by applying the Whittaker functional $\Omega_\tau$ 
to the right-translates of the standard basis elements $\Phi_w^\tau$.
{Recall that the character lattice $X^*(T)$ of
$T$ is isomorphic to the group of cocharacters $X_*(\hat{T})$ of $\hat{T}$. For a character $\lambda \in X^*(T)$,
define $\varpi^\lambda$ to be image of the uniformizer $\varpi$ under the cocharacter 
corresponding to $\lambda$. 
We are interested in the Iwahori-Whittaker functions:
\[\calW_{w}(\tau, g):=\delta^{\frac{1}{2}}(g)\Omega_{\tau^{-1}}(\pi(g)\Phi_w^{\tau^{-1}}),\] where 
$g \in \hat{G}(F)$. As explained in \cite[\S 3]{brubaker2019colored}, this function is determined by its 
values on elements of the form $g= \varpi^\lambda w'$, where $w' \in W$. As in \cite{brubaker.bump.licata},
we will consider the special case when $w'=id$. Further, we fix $\lambda \in X^*(T)$ and $w \in W$, and denote
the resulting function by $\calW_{\lambda,w}(\tau):= \calW_{w}(\tau,  \varpi^{-\lambda})$; under identifications above,
we regard this as a function on $\tau \in T$.} 
\begin{prop}\cite{brubaker.bump.licata,casselman1980unramifiedII}\label{prop:bbl}
The following identities hold as functions on {$T$}:
\begin{enumerate}
\item
For any {$\lambda \in X^*(T)$}
and $w\in W$, $\calW_{\lambda,w}=0$ unless $\lambda$ is {antidominant}.
\item 
{For any antidominant character $\lambda$,}
$\calW_{\lambda,id}=e^\lambda$.
\item
For any $w\in W$ and simple reflection $s_i$, if $s_iw>w$, then
\[\calW_{\lambda,s_iw}=\widetilde{T_i}|_{y\mapsto -q^{-1}}(\calW_{\lambda,w}),\]
where $\widetilde{T_i}$ is defined in Corollary \ref{cor:operators}. In particular,
\[\calW_{\lambda,w}=\begin{cases}\widetilde{T_w}|_{y\mapsto -q^{-1}}(e^\lambda) &\textit{if }\lambda \textit{ is antidominant,}\\
0 &\textit{otherwise.}\end{cases}\]
\item 
For any antidominant {character} 
$\lambda$,
\begin{equation}\label{equ:cs}
\sum_w\calW_{\lambda,w}=\prod_{\alpha>0}(1-q^{-1}e^\alpha)\chi_{w_0\lambda} \/.
\end{equation}
\end{enumerate}
\end{prop}
\begin{rem}
The function $\sum_w  {\calW_{\lambda,w}}$
is called the spherical Whittaker function, and Equation \eqref{equ:cs} is called the Casselman-Shalika formula \cite{casselman1980unramifiedII}.
\end{rem}
\begin{proof}[Proof of Proposition \ref{prop:bbl}] 
The first three properties follow from \cite[Lemma 2 and Theorem 1]{brubaker.bump.licata}. By \cite[Theorem 4]{brubaker.bump.licata}, 
\[\sum_w  {\calW_{\lambda,w}}
=\prod_{\alpha>0}(1-q^{-1}e^\alpha){\widetilde{\partial_{w_0}}(e^\lambda)},\]
where $\widetilde{\partial_{w_0}}$ is defined in Corollary \ref{cor:operators}. 
Therefore, the last property follows from Equation \eqref{equ:weyl}.
\end{proof}
 
Combining with Theorem \ref{thm:characters}, we obtain the following geometric interpretation of the Iwahori-Whittaker function; see also \cite{reeder:padic,su2017k} for related interpretations. 
\begin{cor}\label{cor:geomwhit}
For any antidominant weight $\lambda$ of $T$ and any $w\in W$,
{
\[ \calW_{\lambda,w}=
\chi \left(G/B, \calL_\lambda\otimes MC_{-q^{-1}}'(X(w)^\circ)\right) \/.\]} 
In particular, summing over the Weyl group elements, we get the Casselman-Shalika formula \eqref{equ:cs}
\[\sum_w {\calW_{\lambda,w} }
=\prod_{\alpha>0}(1-q^{-1}e^\alpha)\chi (G/B,\calL_\lambda)\/.\]
\end{cor}

\appendix
\section*{Appendix.  On a formula of Aky{\i}ld{\i}z and Carrell}
\begin{center} 
	with David Anderson
\end{center} 
\label{app.poincare}

\bigskip

\setcounter{equation}{0}
\renewcommand{\theequation}{A.\arabic{equation}}

\setcounter{thm}{0}
\renewcommand{\thethm}{A.\arabic{thm}}

In this appendix, we give a product formula for the Poincar\'e polynomial of a very special type of projective manifold $X$.  Our main examples are smooth Schubert varieties in a (Kac-Moody) complete flag variety $G/B$. In this context, the formula has been reproved many times: when $X=G/B$ is a finite-dimensional flag variety, it was proved by Kostant \cite{kostant1959principal}, Macdonald \cite{macdonald:poincare}, Shapiro and Steinberg \cite{steinberg1959finite}; a refinement to include the case where $X=X(w)$ is a smooth Schubert variety was given by Aky{\i}ld{\i}z and Carrell \cite{akyildiz.carrell:Betti}.  Our aim is to show how this formula fits into the framework of motivic Chern classes.  
Our setup is similar to that of \cite{akyildiz.carrell:Betti}, but where these authors use the language of `$\mathfrak{B}$-regular varieties' and results about the cohomology rings of such spaces, we will deduce the formula directly from basic properties of motivic Chern classes from Theorem \ref{thm:existence} and \cite[ Thm. 4.2]{AMSS19:motivic}, 
using $\mathbb{G}_m$-equivariant localization.

Let $X$ be a $d$-dimensional nonsingular projective variety over an algebraically closed field of characteristic $0$, and suppose $\mathbb{G}_m$ acts with finitely many fixed points.  By the Bia{\l}ynicki-Birula decomposition  \cite{BB} (see also \cite[\S3.1]{Brion}), 
\[
X=\coprod_{p\in X^{\bbG_m}}U(p),
\]
where each cell $U(p)$ is the attracting set of the fixed point $p$, and is isomorphic to some affine space, $U(p)\cong \mathbb{A}^{\ell(p)}$.  For each fixed point $p\in X^{\mathbb{G}_m}$, the weights of $\mathbb{G}_m$ acting on $T^*_pX$ form a multi-set of nonzero integers
\[
  N(p) = \{ n_1(p),\ldots,n_d(p) \}.
\]
It follows from the decomposition that there is a unique ``lowest'' fixed point $p_\circ \in X^{\mathbb{G}_m}$ such that the corresponding cell $U(p_\circ)$ is Zariski open in $X$.  At this point, the weights set $N(p_\circ)$ consists of {\em positive} integers; we will write
\[
n_i:=n_i(p_\circ)
\]
for these weights.  (Likewise, there is a unique ``highest'' fixed point $p_\infty$ such that $N(p_\infty)$ consists of negative integers, but we will not need this.)

We make the following (very restrictive) assumption:

\renewcommand{\theenumi}{\fnsymbol{enumi}}
\begin{enumerate}
\item For each fixed point $p\neq p_\circ$, we have $-1\in N(p)$. \label{app:cond}
\end{enumerate}

Let $A_*(X)_{\mathbb{Q}}$ denote Chow groups (with rational coefficients),  $b_i(X) = \dim A_i(X)_{\mathbb{Q}}$ the $i$th Betti number, and $P(X,q) = \sum_{i=0}^d b_i(X)\,q^i$ the Poincar\'e polynomial.  
%
%
Here is our formula.

\begin{thm}\label{thm:app}
Let $X$ be a nonsingular projective variety, with $\mathbb{G}_m$ acting with finitely many fixed points, satisfying \eqref{app:cond}.  Then
\[
  P(X,q) = \sum_{p\in X^{\bbG_m}} q^{\ell(p)} = \prod_{i=1}^d \frac{ 1 - q^{n_i+1} }{ 1 - q^{n_i} }.
\]
\end{thm}

The idea of the proof is to calculate the Poincar{\'e} polynomial in two ways: on one side this is the same as the integral of the $\lambda_y$ class of $X$, and on the other side the $\lambda_y$ class can be calculated by equivariant localization.

We start with the following simple observation, which is implicit in \cite{AMSS19:motivic}; see also \cite[\S 3.1]{maxim.schurmann:toric}.
 
\begin{lemma}\label{lemma:pt}
Let $U \subseteq X$ be locally closed subset which is isomorphic to an affine cell, $U\cong \mathbb{A}^\ell$. Then $\int_{X} MC_y(U) = (-y)^{\ell}$.
\end{lemma}

\begin{proof} By functoriality, the integral equals $MC_y[U \to pt]= MC_y[(\mathbb{A}^1)^{\ell} \to pt]$.  
We now use the fact that the motivic Chern transformation commutes with exterior products by \cite[Corollary 2.1]{brasselet.schurmann.yokura:hirzebruch} in the non-equivariant case and \cite[ Thm. 4.2]{AMSS19:motivic} equivariantly. 
Then
\[
 MC_y[(\mathbb{A}^1)^{\ell} \to pt]= MC_y[\mathbb{A}^1 \to pt]^{\ell} = (-y)^{\ell} \/, 
\]
where the last equality follows from an explicit calculation of $MC_y[\mathbb{A}^1 \to \pt] = MC_y[\mathbb{P}^1 \to \pt]- MC_y[\pt \to \pt]$.  (The result is independent of the equivariant parameters, even though the classes involved depend on them.)
\end{proof} 

\begin{proof}[Proof of Theorem \ref{thm:app}]
The Bia{\l}ynicki-Birula decomposition makes $X$ into a scheme with a cellular decomposition.  Then its Poincar{\'e} polynomial is given by:
\[
 P(X,q)=\sum_{p\in X^{\bbG_m}} q^{\ell(p)};
\]
see, e.g., \cite[\S3, Corollary~1(iii)]{Brion}.  
By lemma \ref{lemma:pt} and additivity of motivic Chern classes it follows that the integral of the $\lambda_y$ class of $X$ can be calculated as
\[
 \int_X \lambda_y(T^*_X) = MC_y[X \to \pt] = \sum_{p} MC_y[U(p) \to \pt] = \sum_{p} (-y)^{\ell(p)} \/.
\]
Hence, $P(X,q)=\int_X \lambda_{-q}(X)$. On the other hand, since $X$ is smooth, we may compute this integral by localization. Writing $i_p: p \hookrightarrow X$ for the inclusion of a fixed point, we have
\[
 i_p^* MC_y[id: X \to X] = \lambda_y(T^*_{p}X) \/, 
\]
where $T^*_{p}X$ is the cotangent space at $p$.  Then, by Theorem \ref{thm:loc},
\begin{equation}\label{E:2}
\int_{X}\lambda_y(T^*_{X})= \sum_{p \in X^{\mathbb{G}_m}} \frac{\lambda_y(T^*_{p} X)}{\lambda_{-1} (T^*_{p} X)} = \sum _{p \in X^{\mathbb{G}_m}} \prod_{i=1}^d\frac{1+y q^{n_i(p)}}{1-q^{n_i(p)}}  
\end{equation}
in $K_{\mathbb{G}_m}(pt)[y] = \bbZ[q^{\pm 1},y]$, writing $q$ for a coordinate on $\mathbb{G}_m$.

Since $\int_{X}\lambda_y(T^*_{X})\in \bbZ[y]$ does not depend on $q$, we may compute it by any specialization of $q$.  Under the specialization $q=-y$,  only the term corresponding to $p=p_\circ$ survives on the right-hand side of \eqref{E:2}, since for $p\neq p_\circ$, there is some $n_i(p)=-1$.  The $p_\circ$ term specializes to
\[
 \prod_{i=1}^d\frac{1- q^{n_i+1}}{1-q^{n_i}},
\]
and the theorem follows.
\end{proof}

The requirement \eqref{app:cond} is quite restrictive: for instance, in the case $d=2$, $X$ must be either $\mathbb{P}^2$ or a rational ruled surface.  (One checks that for $0<a\leq b$, the rational function $[(1-q^{a+1})(1-q^{b+1})]/[(1-q^a)(1-q^b)]$ is a polynomial if and only if $(a,b)$ is one of $(1,1)$, $(1,2)$, or $(2,3)$, and the last cannot occur for a projective surface.)

However, Schubert varieties are examples.  Here we consider a Kac-Moody group $G$, with Borel subgroup $B$ and maximal torus $T$, and corresponding simple roots $\{\alpha_i\}$.  If $\beta=\sum m_i \alpha_i$ is any root, its {\it height} is
\[
  \h(\beta) := \sum m_i .
\]
Continuing notation from the main part of the paper, for each element $w$ of the Weyl group $W$, we have Schubert cells $X(w)^\circ = BwB/B \cong \mathbb{A}^{\ell(w)}$, and Schubert varieties $X(w)=\overline{X(w)^\circ}$.  Finally, for any $v\leq w$, we write $e_v \in X(w)$ for the corresponding $T$-fixed point.  (Our reference for Kac-Moody groups is Kumar's book \cite{KumarKM}.  The hypothesis that $G$ be reductive, imposed in the main part of the paper for Iwahori-Whittaker theory, is not required here.)

\begin{cor}\label{cor:app}
Let $X(w) \subseteq G/B$ be a smooth Schubert variety in a Kac-Moody flag variety.  Then
\[
  P(X(w),q) := \sum_{v \leq w} q^{\ell(v)} = \prod_{\beta>0,\, s_\beta \leq w} \frac{ 1 - q^{\h(\beta)+1} }{ 1 - q^{\h(\beta)} }.
\]
\end{cor}

\begin{proof}
Let $\check\rho\colon \mathbb{G}_m \to T$ be a cocharacter such that $\langle \alpha_i, \check\rho \rangle = 1$ for each simple root $\alpha_i$.  Then for any root $\beta$, we have $\h(\beta) = \langle \beta, \check\rho \rangle$.  Under the restriction homomorphism $K_T(pt) \to K_{\mathbb{G}_m}(pt)=\bbZ[q^{\pm 1}]$, one has $e^\beta \mapsto q^{\h(\beta)}$.

By a standard calculation (see, e.g. \cite[Theorem 8.1]{AMSS19:motivic}), the weights for $T$ acting on the cotangent space $T^*_{e_v}X(w)$ are $\{ v(\beta) \,|\, \beta>0 \text{ and } vs_\beta \leq w\}$, so the corresponding weights for the $\mathbb{G}_m$ action are $\{ \h(v(\beta)) \,|\, \beta>0 \text{ and } vs_\beta \leq w\}$.  Since these are all nonzero, we have $X^{\mathbb{G}_m} = X^T$.

To apply the theorem, it suffices to show that whever $v\neq id$, we have $\h(v(\beta))=-1$ for some $\beta>0$ such that $vs_\beta\leq w$.  Find a simple root $\alpha_i$ such that $v^{-1}(\alpha_i)<0$, and let $\beta = -v^{-1}(\alpha_i)$.  Then $v(\beta)<0$, so $vs_\beta\leq v \leq w$, and therefore $v(\beta)$ is a cotangent weight.  But also $v(\beta) = -\alpha_i$ has $\h(v(\beta))=-1$.
\end{proof}
In the case where $G$ is finite-dimensional, if one takes $w=w_0$, then $X(w) = G/B$ and Corollary~\ref{cor:app} gives the classical formula for the Poincar{\'e} polynomial of $G/B$:
\[
 P(G/B,q)=\sum_{v \in W}q^{\ell(v)}=\prod_{\alpha>0}\frac{1-q^{\h(\alpha)+1}}{1-q^{\h(\alpha)}} \/.
\]
In this case, our proof is similar to one given by Macdonald \cite{macdonald:poincare}, with the interesting twist that instead of the localization formula, Macdonald uses the Lefschetz fixed point formula for the left multiplication by an element of $T$. 
\begin{rem} For Schubert varieties in finite flag manifolds, the theorem holds over an algebraically closed field of arbitrary characteristic. This is because Schubert classes form a basis of the Chow ring of $G/B$ \cite[\S 3, Corollary 1(iii)]{Brion}, thus the Poincar{\'e} polynomial is   independent of the choice of the field (it only depends on the Weyl group). 
\end{rem}
\bibliographystyle{halpha}
\bibliography{whittakerbib}
\end{document}